\documentclass{amsart}
\usepackage[mathscr]{eucal}
\usepackage{amssymb, amsmath,array, amscd}
\usepackage[OT2,T1]{fontenc}

\newtheorem{theorem}{Theorem}[section]
\newtheorem{lemma}[theorem]{Lemma}
\newtheorem{rem}[theorem]{Remark}
\newtheorem{claim}[theorem]{Claim}
\newtheorem{prob}{Problem}

\theoremstyle{definition}

\newtheorem{definition}[theorem]{Definition}

\newtheorem{proposition}[theorem]{Proposition}

\newtheorem{cor}[theorem]{Corollary}

\numberwithin{equation}{section}

\newtheorem{main}{Theorem}

\newcommand{\M}[1]{\mathfrak{M}_{0,{#1}}}   

\def\O{\mathcal{O}}

\def\C{\mathbb C}
\def\si{\sigma}
\def\var{\varphi}
\def\a{\alpha}
\def\be{\beta}
\def\d{\mathbb{D}}
\def\N{\mathbb{N}}
\def\om{\omega}

\def\X{\mathcal X}

\def\G{\mathcal G}
\def\H{\mathcal H}
\def\Te{\Theta}
\def\fa{\mathcal{F}}
\def\p{\mathbb{P}}
\def\Ga{\Gamma}
\def\sub{\subset}
\def\ov{\overline}
\def\g{\gamma}
\def\pa{\partial}
\def\M{\mathcal{M}}
\def\la{\lambda}
\def\U{\mathcal{U}}
\def\I{\mathcal{I}}
\def\emp{\emptyset}
\def\Si{\Sigma}
\def\te{\theta}

\def\fol{\mathbb{F}ol}

\def\wt{\widetilde}

\begin{document}

\title[Pull-back components]{Pull-back components of the space of foliations of codimension $\ge2$}



\author[W. Costa e Silva]{W. Costa e Silva}
\address{IMPA, Est. D. Castorina, 110, 22460-320, Rio de Janeiro, RJ, Brazil}

\email{wancossil@gmail.com}

\author[A. Lins Neto]{A. Lins Neto}
\address{IMPA, Est. D. Castorina, 110, 22460-320, Rio de Janeiro, RJ, Brazil}

\email{alcides@impa.br}


\date{}


\begin{abstract}
We present a new list of irreducible components for the space of k-dimensional  holomorphic foliations on  $\mathbb P^{n}$, $n\geq3$, $k\ge2$.  They are associated to pull-back of dimension one foliations on $\mathbb P^{n-k+1}$ by non-linear rational maps.
\end{abstract}

\maketitle
\setcounter{tocdepth}{1}
\tableofcontents \sloppy

\section{Introduction}\label{ss:1}
A codimension $q$ singular holomorphic foliation $\mathcal F$ on a complex manifold $M$, $dim_{\mathbb C}M\geq q+1$, can be defined locally by the following data: 

\begin{enumerate}
\item a covering $\mathcal U = (U_{\alpha})_{{\alpha \in A}}$ of $M$ by open sets.
\item a collection $(\eta_{\alpha})_{{\alpha \in A}}$ of  $q$-forms, $\eta _{\alpha} \in \Omega^q_{U_{\alpha}}$, having the following properties:
\begin{enumerate}
\item Local decomposability: Given  $p\in U_{\alpha}$ such that $\eta _{\alpha}(p)\neq0$ there exist a neighborhood $U$ of $p$, $U\subset U_{\alpha}$, and holomorphic $1$-forms 
$\omega_1,\cdots,\omega_q$ such that  $\eta_{\alpha}|_{U}= \omega_1 \wedge \cdots \wedge \omega_q;$\\
\item The decomposition of $\eta _{\alpha}$ satisfy the Frobenius integrability condition
$d  \omega_i \wedge  \eta_{\alpha} =0 \quad \text{ for every } i = 1 , \ldots, q$.
\end{enumerate}
\item a multiplicative cocycle $G:=(g_{\alpha\beta})_{U_{\alpha}\cap U_{\beta}\neq0}$ such that $\eta_{\alpha}=g_{\alpha\beta}\eta_{\beta}$.
\end{enumerate}

When $dim_\C M=n$ we say also that the foliation is $(n-q)$-dimensional.

The line bundle induced  by the cocycle $G$ is denoted by $N_{\mathcal F}$ and called the normal bundle of $\mathcal F$. The family $(\eta_{\alpha})_{{\alpha \in A}}$, defines a global section $\eta \in H^0(M,\Omega^q_M\otimes N_{\mathcal F})$. 
The analytic subset, $Sing({\mathcal F}):=\bigcup_\alpha\{p\in {M}|\eta_{\alpha}(p)=0\}$ is the singular set of ${\mathcal F}$. We can always assume that $Sing({\mathcal F})$ has codimension greater or equal than two.

If $M=\mathbb P^{n}$, the $n$-dimensional complex projective space, then a codimension $q$ singular foliation $\mathcal F$ is given by a global section of 
$H^{0}(\mathbb P^{n},\Omega^q_{\mathbb P^{n}}\otimes \mathcal O_{\mathbb P^n}(\Theta+q+1))$, where $\Theta$ (called the degree of $\mathcal{F}$)  is  the degree of the divisor of tangencies of the foliation with a generic $\mathbb{P}^{\,q}$ linearly embedded in $\mathbb P^n$.  
On the other hand, a global section of
$H^{0}(\mathbb P^{n},\Omega^q_{\mathbb P^{n}}\otimes \mathcal O_{\mathbb P^n}(\Theta+q+1))$ can be represented, in homogeneous coordinates, by a polynomial $q$-form $\eta$ on $\mathbb C^{n+1}$ with homogeneous coefficients of degree
$\Theta+1$ and satisfying $i_R\eta=0$, where $R=\sum_{i=0}^{n}z_i\frac{\partial}{\partial z_i}$ is the  radial vector field. In fact, the $q$-form $\eta$ defines the foliation $\Pi^{\ast}(\mathcal F)$, where $\Pi:\mathbb C^{n+1}\backslash \{0\}\to\mathbb P^n$ is the canonical projection. For more details see \cite{jou} and  \cite{ln}.

The projectivisation of the set of integrable $q$-forms as above, defining in homogeneous coordinates $k$-dimensional foliations on $\p^n$ of degree $\Te$, will be denoted by $\fol(\Te;k,n)$. 
Due to the integrability condition, $\fol(\Te;k,n)$ is a quasi projective algebraic subset of $\p H^{0}(\mathbb P^{n},\Omega^q_{\mathbb P^{n}}\otimes \mathcal O_{\mathbb P^n}(\Theta+n-k+1))$. 
A natural question that arises is the following:

\vskip0.2cm
\noindent \textbf{{\underline{Problem}: \emph{Identify and classify the irreducible components of \break $\mathbb{F}{\rm{ol}}\left(\Theta
;k,n\right)$  on ${\mathbb P^n}$, such that $\Theta
\geq 0$ and $n\geq k+1$}}}.
\vskip0.2cm
The classification of the irreducible components of $\mathbb{F}{\rm{ol}}\left(0;k,n\right)$ was given in \cite[Th. 3.8 p. 46]{cede} (a k-dimensional  foliation of degree zero on $\mathbb P^n$ is a rational fibration defined by a linear projection from $\mathbb P^n$ to $\mathbb P^{n-k}$). The classification of the irreducible components of $\mathbb{F}{\rm{ol}}\left(1;k,n\right)$, which require more details to be explained here, was given in \cite[Th. 6.2 and Cor. 6.3 p. 935-936]{lpt}. The situation for $\Theta\geq2$ remains wide open, unless in the case of codimension one and degree $\Theta=2$  (see\cite{cln}).

Usually, the space of codimension one foliations on $\p^n$ of degree $k$ is denoted by $\mathbb{F}{\rm{ol}}\left(k,n\right)$. The study of irreducible components of these spaces has been initiated by Jouanolou in \cite{jou}, where  the irreducible components of  $\mathbb{F}{\rm{ol}}\left(k,n\right)$ for $k=0$ and $k=1$ are described. The case of degree two was studied in the  paper \cite{cln}, where the authors proved that $\mathbb{F}{\rm{ol}}\left(2,n\right)$ has six irreducible components, which can be described by geometric and dynamic properties of a generic element. 
In the general case, degree $\ge3$, one can exhibit some kind of list of irreducible components in every degree, but it is not known if this list is complete. 

When we study the components of the space $\mathbb{F}{\rm{ol}}\left(k,n\right), n\geq 3$ we perceive that there are families of irreducible components in which the typical element is a pull-back of a foliation on $\mathbb{P}^2$ by a rational map. More precisely, the situation is as follows: given a generic rational map $f: \mathbb{P}^n  -\to \mathbb{P}^2$ of degree $\nu\geq1$, and a degree $d$ foliation $\mathcal G$ of $\p^2$, then it can be associated to the pair $(f,\mathcal G)$ the pull-back foliation $f^*\mathcal G$ on $\p^n$.
If $f$ and $\mathcal G$ are generic then the degree of $f^*\mathcal G$ is $\nu(d+2)-2$, as proved in \cite{clne}.
Denote by $PB(d,\nu,n)$ the closure in $\mathbb{F}{\rm{ol}}\left(\nu(d+2)-2,n\right)$, $n \geq 3$ of the set of foliations $f^{\ast}\mathcal G$ as above.
The main result of \cite{clne} is the following:
 
\begin{theorem}\label{teo1.1}  \cite{clne} The set $PB(d,\nu,n)$ is a unirational irreducible component of\break $\mathbb{F}{\rm{ol}}\left(\nu(d+2)-2,n\right)$ for all $n \geq 3$, $\nu\geq1$  and $d \geq 2$.
\end{theorem}

The case $\nu=1$, of linear pull-backs, was proven in \cite{caln}, whereas the case $\nu>1$, of nonlinear pull-backs, was proved in \cite{clne}. 
In the paper \cite{cupe} the authors were able to prove the existence of irreducible components of linear pull-back foliations in arbitrary codimension. This was obtained by using the technics of stability of the tangent sheaf of the foliation. Linear Pull-back foliations have tangent sheaf which is locally free and therefore are stable. On the other hand, irreducible components of the space of foliations where a typical element is a non linear pull-back was known only for the codimension one situation.  The main purpose of this work is to show that at least in some cases there exist non linear pull-back type components of $\mathbb{F}{\rm{ol}}\left(\Theta;k,n\right)$, for several values of $\Theta$.
Let us describe, briefly, the type of pull-back foliation that we shall consider.\\ 
From now on we will always assume $n\geq3$. Let $f: {\mathbb P^n}  -\to \mathbb{P}^{\,q+1}$ be a rational map, represented in the homogeneous coordinates $z\in \mathbb{C}^{n+1}$ and
$x \in \mathbb C^{\,q+2}$ by \break $\wt{f}=(F_{0},F_{1},...,F_{q+1})$, where the $F_{i's}$ are homogeneous polynomials of degree $\nu$, without common factors.
Let $\mathcal G$ be a 1-dimensional foliation on $\mathbb P^{\,q+1}$ of degree $d$ and \break $\Pi_{q+1}:\C^{q+2}\setminus\{0\}\to\p^{\,q+1}$ be the canonical projection. Then $\Pi_{q+1}^{\ast}\mathcal G$ can be defined by a $q$-form $\Omega=i_Ri_XdV$, where $R$ is the radial vector field, $X=\sum_{i=0}^{q+1}P_i(x)\frac{\partial}{\partial x_i}$ is a homogeneous polynomial vector field of degree $d$ and
$dV=dx_0\wedge\dots\wedge dx_{q+1}$.  
If $f$ and $\mathcal G$ are generic, in a sense that will be precised later, then the form $\wt{f}^*\,\Omega$ represents in homogeneous coordinates the foliation $f^*\G:=\fa$, of codimension $q$ on $\p^n$.
It can be checked that the degree of $\fa$ is $\Theta(\nu,d,q)=(d+q+1)\nu-q-1$.
Let $PB(\nu,d,k,n)$, $k=n-q$, be the closure in $\fol(\Te;k,n)$, $\Te=\Te(\nu,d,q)$, of the set of this kind of foliations. In this paper we are able to prove the following:

\begin{main}\label{teob} The set $PB(\nu,d,k,n)$ is a unirational irreducible component of $\mathbb{F}{\rm{ol}}\left(\Theta;k,n\right)$ for all $\nu\geq2$, $d \geq 2$, $k\ge2$ and $n\geq3$. 
\end{main}
Let us observe that a generic element in the set $PB(\nu,d,k,n)$ has no algebraic invariant leaf because a generic foliation by curves of degree $\geq2$ also does not have \cite{lnsoares}.
It is worth pointing out that we recover Theorem \ref{teo1.1} in the case $n\ge3$ and $k=n-1$. On the other hand, for $n\geq4$ we present new families of irreducible components that were not known.

The proof of theorem \ref{teob} will be done first in the case $k=2$, that is for two dimensional foliations on $\p^n$ that are pull back of one dimensional foliations on $\p^{n-1}$. The general case, $k\ge3$, will be done in \S\,\ref{ss:34}.  

\section{Preliminaries}

\subsection{One dimensional foliations on $\mathbb P^m$, $m\ge2$}\label{ss:21}

Let us recall some definitions and basic facts about one dimensional foliations that we will use.
Let $\mathcal X$ be a germ at $0\in\mathbb C^m$ of holomorphic vector field, $0\in Sing(\mathcal X)$ and denote by $\lambda_1,\dots,\lambda_{m}\in \mathbb C$ the eigenvalues of the linear part of $D\mathcal X(0)$. The germ of foliation defined by $\mathcal X$ will be denoted by $\mathcal G_{\mathcal X}$. We say that $0$ is a non-degenerate singularity of $\mathcal X$ if $\lambda_j\neq0$ for all $j=1,...,m$. In this case, $0$ is an isolated singularity of $\X$. In this case, the singularity is hyperbolic if all the quotients ${\lambda_i}/{\lambda_j}$, $i\ne j$, are not real. The singularity is of Kupka type if $tr(D\mathcal X(0))=\lambda_1+\dots+\lambda_m\neq0$. When $0$ is a hyperbolic singularity of $\mathcal X$ then there are exactly $m$ germs of analytic $\mathcal G_{\mathcal X}$-invariant curves, say $\Gamma_1,\dots,\Gamma_m$, through $0\in\mathbb C^m$, called the separatrices of  $\mathcal G_{\mathcal X}$ through $0$. Moreover, $\Gamma_j$ is smooth and tangent to the eigenspace associated to $\lambda_j$, $1\le j\le m$ (see\cite{lnsoares}).  

Let us denote by $\mathbb{F}{\rm{ol}}\left(d;1,m\right)$ the space of one dimensional foliations on $\mathbb P^{m}$. From \cite{ln3} and \cite{lnsoares} we know that, given $m\geq2$ and $d\geq2$, there is an open and dense subset $\mathcal M(d)\subset\mathbb{F}{\rm{ol}}\left(d;1,m\right)$ such that any $\mathcal G\in\mathcal M(d)$ satisfies: 

\begin{enumerate}
\item $\mathcal G$ has exactly $N=\frac{d^{m+1}-1}{d-1}$ singularities, all of them hyperbolic
\item $\mathcal G$ has no algebraic invariant curve
\item All singularities are of Kupka-type
\end{enumerate}
 
When $m=2$ the property $(2)$ implies that the Zariski closure of any one dimensional leaf of $\mathcal G$ is $\mathbb P^2$. A natural question is what happens if $m\geq3$. In this case, it is not known if there exists an open and dense subset of $\mathbb{F}{\rm{ol}}\left(d;1,m\right)$ with a similar property. However, it is known that there exists a generic subset $\mathcal M_g(d)\subset\mathcal M(d)\subset\mathbb{F}{\rm{ol}}\left(d;1,m\right)$ such that any foliation $\mathcal G\in\mathcal M_g(d)$ has no algebraic invariant subset of positive dimension (see\cite{cope}). In particular, if $\mathcal G\in\mathcal M_g(d)$   then the Zariski closure of any leaf of $\mathcal G$ is $\mathbb P^{m}$.

\subsection{Rational maps}\label{ss:22}
Let $f : \p^n -\to  \p^m$ be a rational map, and let $\tilde{f}: {\mathbb C^{n+1}} \to {\mathbb C^{m+1}}$ its natural lifting in homogeneous coordinates.

\begin{definition}\label{d:21}
We denote by $RM\left(n,m,\nu\right)$ the set of maps $\left\{f: \mathbb P^n  -\to   \mathbb P^m\right\}$ of degree $\nu\geq 2$ given by
$\tilde f=\left(F_{0},F_{1},...,F_{m}\right)$  where the $F_{js}$, are homogeneous polynomials of degree $\nu$ and without common factors.
\end{definition}

The \emph{indeterminacy locus} of $f$ is, by definition, the set $I\left(f\right)=\Pi_{n}\left(\tilde{f}^{-1}\left(0\right)\right)$, where $\Pi_{n}:\mathbb C^{n+1}\backslash \{0\}\to\mathbb P^{n}$ is the canonical projection. Observe that the restriction $f|_{\mathbb P^n \backslash I\left(f\right)}$ is holomorphic. 

\begin{definition}\label{generic} We say that $f  \in RM\left(n,m,\nu\right)$ is $generic$ if for all  \break$p \in$ $\tilde{f}^{-1}\left(0\right)\backslash\left\{0\right\}$ we have $dF_{0}\left(p\right)\wedge dF_{1}\left(p\right)\wedge...\wedge dF_{m}\left(p\right) \neq 0.$  
\end{definition}

This is equivalent to saying that $f  \in RM\left(n,m,\nu\right)$ is $generic$ if $I(f)$ is the transverse intersection of the $m+1$ hypersurfaces $\Pi_n(F_{i}=0)$ for $i=0,...,m.$  If $f$ is generic, $n=m+1$ and $deg(f)=\nu$ then $I\left(f\right)$ consists of $\nu^{m+1}$ distinct points, by Bezout's theorem. On the other hand, if $n>m+1$ then $I(f)$ is a connected smooth complete intersection of degree $\nu^{m+1}$.

Let $U(f)=\mathbb P^n\backslash I(f)$, $P(f)$ be the set of critical points of $f$ in $U(f)$ and $C(f)=f(P(f))$ be the set of the critical values of $f$. If $f$ is generic, then $\overline{P(f)}\cap I(f)=\emptyset$, so that  $\overline{P(f)}=P(f)\subset U(f)$ (where $\overline A$ denotes the closure of $A\subset \mathbb P^n$ in the usual topology). Since $P(f)=\{p\in U(f);rank(df(p)\leq m-1\}$, it follows that $P(f)$ is a proper algebraic subset of $\mathbb P^n$ and $C(f)$ is a proper algebraic subset of $\p^{m}$. 
The set of generic maps of degree $\nu$ will be denoted by $Gen\left(n,m,\nu\right)$. 
\begin{proposition}
$Gen\left(n,m,\nu\right)$ is a Zariski open and dense subset of $RM\left(n,m,\nu\right)$.
\end{proposition}

\subsection{Generic pairs}\label{ss:23}
\begin{definition} Let $f$ be an element of $Gen\left(n,m,\nu\right)$. We say that  
$\mathcal G \in \mathcal M(d)\sub \fol(d;1,m)$ is in generic position with respect to $f$ if $Sing\left(\mathcal G\right)\cap C(f)=\emptyset$. In this case we will say that $(f,\mathcal G)$ is a generic pair.
\end{definition}

Set $\mathcal W=\{\fa; \fa = f^*\G$ and $(f,\mathcal{G})\in Gen\left(n,m,\nu\right)\times \mathcal M(d)$ is a generic pair$\}$.
We remark that $\mathcal W$ is an open and dense subset of $PB(\nu,d,k,n)$, where $k=n-m+1$. In fact it is a real Zariski open set.

As we have seen before, $f^*\G$ can be represented in homogeneous coordinates by $\wt{f}^*\Omega$, where
\[
\wt{f}=(F_0,...,F_m)
\]
represents $f$ and 
\[
\Omega=i_Ri_XdV=\sum_{i\neq k} (-1)^{i+k+1} x_{k}P_{i}dx_0\wedge...\wedge \widehat{dx_i}\wedge...\wedge\widehat{dx_k}\wedge...\wedge dx_{m} 
\]
represents $\G$.  The pull back foliation, $f^{*}(\mathcal G)$, is then defined in homogeneous coordinates by the $(m-1)$-form
\[
\eta_{[f,\mathcal G]}=\left[
 (-1)^{i+k+1}\sum_{i\neq k} F_{k}.(P_{i}\circ\tilde{f})\space dF_0\wedge...\wedge \widehat{dF_i}\wedge...\wedge\widehat{dF_k}\wedge...\wedge dF_{m}\right]\,\,.
\]
Since the $F_{k's}$ are homogenous of degree $\nu$ and $P_{i's}$ homogeneous of degree $d$, the coeficients of $\eta_{[f,\G]}$ are homogeneous of degree $(d+m)\nu-m+1$. From these considerations we get:

\begin{proposition}\label{graupargenerico}
If $\fa=f ^{\ast}\mathcal G$ where $(f,\mathcal G)$ is a generic pair, then the degree of $\fa$ is \break$\Theta(\nu,d,n)=(d+m)\nu-m$. 
\end{proposition}

\subsection{The Kupka set}\label{ss:24}

The purpose of this section is to expose the main facts about the Kupka set of two dimensional foliations that will be used in the first part of the proof of theorem \ref{teob} (see \S\, \ref{ss:33}).  

Let $\eta$ be a germ at $(\C^{n},p)$ of holomorphic $(n-2)$-form. Since $d\eta$ is a $(n-1)$-form there exists a germ at $(\C^n,p)$ of vector field $Z$ such that $d\eta=i_Z\nu$, where $\nu=dz_1\wedge...\wedge dz_n$. The vector field $Z$ is called the {\it rotational} of $\eta$ with respect to $\nu$. We will denote $Z:=rot_\nu(\eta)$. If $\wt\nu$ is another non-vanishing $n$-form then there is a vector filed $\wt{Z}=rot_{\nu'}(\eta)$ such that $d\eta=i_{\wt{Z}}\nu'$. Since $\nu'=u.\,\nu$, where $u(p)\ne0$, we have $rot_{\nu'}(\eta)=u.\,rot_\nu(\eta)$. In particular, $rot_\nu(\eta)$ and $rot_{\wt\nu}(\eta)$ define the same germ of one dimensional foliation.
We say that $p$ is singularity of {\it Kupka type} of $\eta$ if $\eta(p)=0$ and $rot_\nu(\eta)(p)\ne0$. This condition is equivalent to
$\eta(p)=0$ and $d\eta(p)\ne0$. 
Therefore, if $\wt\eta=u.\,\eta$, where $u(p)\ne0$, then
\[
\eta(p)=0\,\,\text{and}\,\,d\eta(p)\ne0\,\,\iff\,\,\wt\eta(p)=0\,\,\text{and}\,\,d\wt\eta(p)\ne0\,\,.
\]
In other words, the concept depends only of the foliation defined by $\eta$, $\fa_\eta$.
In particular, it can be extended to foliations on complex manifolds.
\begin{definition}
Let $\fa$ be a two dimensional foliation on a complex manifold $M$. We say that $p\in M$ is a singularity of Kupka type of $\fa$ if $\fa$ is represented in a neighborhood of $p$ by a $(n-2)$-form with a Kupka singularity at $p$. The set of singularities of Kupka type of $\fa$ will be denoted by $K(\fa)$.
\end{definition}
The following result is a special case of more general one proved in \cite{medeiros} (see also \cite{ln3}):

\begin{proposition}\label{p:27}
({\it Local product structure.}) 
Let $\eta$ be a germ at $(\C^n,p)$ of integrable $(n-2)$-form and $Z=rot(\eta)$. Assume that $Z(p)\ne0$. Then there exists a coordinate system around $p$, say  $\phi=(y,t):U\to\mathbb C^{n-1}\times \mathbb C$, $y:U\to\mathbb C^{n-1}$, $t:U\to\mathbb C$ such that $\phi(p)=(0,0)$ and when expressed in this coordinate system $\eta$ depends only of $y$. In other words, we can write   $\eta=i_Ydy_{1}\wedge\dots\wedge dy_{n-1}$ where $Y$ is a vector field of the form $Y=\sum_{j=1}^{n-1}Y_{j}(y)\frac{\partial}{\partial y_j}$ and $Y(0)=0$. Moreover, $tr(DY(0))\neq0$, that is $0$ is a singularity of Kupka type of $Y$.
\end{proposition}

\begin{rem}
{\rm It follows from proposition \ref{p:27} that $\fa_\eta$ is equivalent to a foliation which is a product of a one dimensional non-singular foliation by a the foliation induced by $Y$. By this reason, the vector field $Y$ is called the normal type of the Kupka set of $\fa_\eta$ at $p$.

If the normal type $Y$ has an isolated singularity at $0$, $det(DY(0))\ne0$, then in the coordinate system $(y,t)$ of proposition \ref{p:27}, the germ of $K(\eta)$ at $p$ is the smooth curve $(y=0)$. Moreover, the normal type is constant along this curve.
An example is when $Y$ has a non-degenerate singularity at $p$. In this case, we say that $p$ is a non-degenerate Kupka singularity of $\fa_\eta$.

In the global case, if a foliation $\fa$ on a complex compact manifold $M$ has a non-degenerate Kupka singularity $p$ then the irreducible component of $Sing(\fa)$ that contains $p\,$ is a compact complex curve $\Ga$ such that $\Ga\setminus K(\fa)$ is finite, say
$\Ga\setminus K(\fa)=\{p_1,...,p_r\}$, and the normal type of $\fa$ is constant along $\Ga\setminus\{p_1,...,p_r\}$.}
\end{rem}
The Kupka set is locally stable under deformations, as explained below.

Let $(\eta_t)_{t\in D_r}$ be a holomorphic one parameter family of integrable $(n-2)$-forms defined in a neighborhood of a closed ball $\ov{B}\sub\C^n$, $D_r=\{\tau\in\C\,|\,|\tau|<r\}$. Assume that $K(\eta_0)$ contains a holomorphic curve $\Ga$ with the following properties:
\begin{itemize}
\item[(i).] $\Ga':=\Ga\cap\ov{B}$ is biholomorphic to a closed disc, $\ov\d\sub\C$, and cuts transversely the boundary $\pa B$ of $\ov B$.
\item[(ii).] The normal type of $\eta_0$ along $\Ga'$, say $Y_0$, is non-degenerate. 
\end{itemize}
The following result is a consequence of \cite{medeiros}:

\begin{proposition}\label{p:29}
In the above situation there exists $r'<r$, a $C^\infty$ isotopy $\Phi\colon\Ga'\times D_{r'}\to\ov{B}$ and a holomorphic family of germs at $0\in\C^{n-1}$ of holomorphic vector fields $(Y_\tau)_{\tau\in D_{r'}}$, with the following properties:
\begin{itemize}
\item[(a).] $\Phi_0(\Ga')=\Ga'$ and for any $\tau\in D_{r'}$ then $\Phi_\tau(\Ga')\sub K(\eta_\tau)$, where $\Phi_\tau(z)=\Phi(z,\tau)$.
\item[(b).] $Y_\tau$ is the normal type of $\eta_\tau$ along $\Phi_\tau(\Ga')$ and has a non-degenerate singularity at $0\in \C^n$, $\forall\tau\in D_{r'}$.
\end{itemize}
\end{proposition}

\vskip.1in

Next we will describe the Kupka set of a foliation $\fa=f^*\G$ on $\p^n$, where $f\colon\p^n-\to\p^{n-1}$, $\G\in\M(d)$ and $(f,\G)$ is a generic pair.
Let $Sing(\G)=\{q_1,...,q_N\}$ and recall that, by definition, $q_1,...,q_N$ are all singularities of Kupka type of $\G$.
Set $V_{q_i}=\overline{f^{-1}(q_i)}$, $1\le i\le N$. 
Observe first, that  $V_{q_i}$ is a smooth algebraic curve of $\mathbb P^n$ that contains $I(f)$.
In fact, it is the transverse intersection of $n-1$ hypersurfaces because the pair $(f,\G)$ is generic.
For fixed $i\in\{1,...,N\}$ let $Y_i$ be a holomorphic vector field representing $\G$ in a neighborhood of $q_i$. Since $\G\in \M(d)$, $q_i$ is a non-degenerate singularity of $Y_i$ of Kupka type; $tr(DY_i(q_i))\ne0$.
\begin{lemma}\label{l:210}
{\rm In the above situation we have:
\begin{itemize}
\item{} $V_{q_i}\setminus I(f)\,\sub K(f^*(\G))$, $\forall i$.
\item{} The normal type of $f^*(\G)$ along $V_{q_i}\setminus I(f)$ is equivalent to $Y_i$, $\forall i$.
\end{itemize}}
\end{lemma}

{\it Proof.}
Fix $i\in\{1,...,N\}$ and $p\in V_{q_i}\setminus I(f)$. Since $q_i$ is a regular value of $f|_{U(f)}$, then $f$ is a submersion in a neighborhood of $p$.  Hence, there exist local analytic coordinate systems $(U,y,t)$,
$y:U\to\C^{n-1}$, $t:U\to\C$, and $(V,u)$, $u:V\to\C^{n-1}$, at $p$ and $q_i=f(p)$ respectively, such that $f(U)\sub V$, $u(q_i)=0\in\C^{n-1}$, $f(y_1,y_2,...,y_{n-1},t)=(y_1,y_2,...,y_{n-1})$ and $\G$ is represented in $V$ by a vector field $Z$ holomorphically equivalent near $0\in\C^{n-1}$ to $Y_i$ near $q_i$. In particular, $det(DZ(0))\ne0$ and $tr(DZ(0))\ne0$.
Therefore, if $Z=\sum_{j=1}^{n-1}Z_j(u)\frac{\pa}{\pa u_j}$ then $\G$ is represented on $V$ by the $(n-1)$-form
$\om=i_Z\,du_1\wedge...\wedge du_{n-1}$. Since $f(y,t)=y$ the form $\eta:=f^*(\om)$ has essentialy the same expression as $\om$:
\[
\eta=f^*\om=i_{Z(y)}\,dy_1\wedge...\wedge dy_{n-1}\,\,.
\]
In particular, we get $d\eta(p)=tr(DZ(0))\,dy_1\wedge...\wedge dy_{n-1}$ and $rot(\eta)=tr(DZ(0))\,\frac{\pa}{\pa t}\ne0$.
Therefore, $p\in K(f^*(\G))$ and the normal type of $f^*(\G)$ is equivalent to the germ of $Z$ at $0$, which proves lemma \ref{l:210}.
\qed
   
\subsection{Conic singularities}\label{ss:25}
The purpose of this section is to describe the so called {\it conic} singularities in the case of two dimensional foliations.
We begin by proving that a pull-back foliation $f^*(\G)$, where $f\colon\p^n-\to\p^{n-1}$ and $(f,\G)$ is a generic pair, has a conic structure near a point $p\in I(f)$.

Fix a point $p\in I(f)$ and let $\wt{p}\in\Pi_n^{-1}(p)$. Without lost of generality we can assume that $\wt p=(1,0,...,0)\in\C^{n+1}$.
Let $\wt f=(F_0,...,F_{n-1})\colon\C^{n+1}\to\C^n$ be the homogeneous lifting of $f$. Since $f$ is generic, $\wt{f}$ is a submersion at $\wt{p}\,$; there exists a local coordinate system $(U,x=(x_0,...,x_n)\in \C^n)$ around $\wt p$ such that 
\[
\wt f(x_0,x_1,...,x_n)=(x_1,...,x_n)\,\,\implies\,\,f[1:x_1:...:x_n]=[x_1:...:x_n]\in\p^{n-1}\,\,.
\]
In other words, in the affine chart $[1:x]\simeq x\in\C^n\sub\p^n$, the map $f$ is the canonical projection
$x\in\C^n\setminus\{0\}\mapsto [x]\in\p^{n-1}$. In particular, the pull-back foliation $f^*(\G)$ is defined in these coordinates by an integrable  $(n-2)$-form $\eta$ with homogeneous coefficients of degree $d+1$ and such that $i_R\eta=0$, $R$ the radial vector field on $\C^n$.
In fact, the form $\eta$ defines $\G$ in homogeneous coordinates.

In the next lemma we study the rotational of a form $\eta$ defining in homogeneous coordinates a foliation on $\p^{n-1}$.
Let $\G$ be a foliation of degree $d\ge2$ on $\p^{n-1}$ and $\eta$ be a $(n-2)$-form on $\C^n$, with homogeneous coordinates of degree $d+1$, defining $\Pi_{n-1}^*(\G)$. Let $X:=rot(\eta)$; $d\eta=i_X\nu$, $\nu=dx_1\wedge...\wedge dx_n$. 

\begin{lemma}\label{l:211}
{\rm Assume that all singularities of $\G$ are non-degenerate. Then $0\in\C^n$ is an isolated singularity of $X$ if, and only if, $\G\in\M(d)$.}
\end{lemma}

{\it Proof.}
We will use the identity:
\begin{equation}\label{eq:21}
(n+d-1)\,\eta=i_R\,d\eta\,\,.
\end{equation}
Let us prove (\ref{eq:21}). Since the coefficients of $\eta$ are homogeneous of degree $d+1$ we have $L_R\eta=(n+d-1)\,\eta$. On the other hand, from $i_R\eta=0$ we get
\[
(n+d-1)\,\eta=L_R(\eta)=i_R\,d\eta+d(i_R\eta)=i_R\,d\eta\,\,.
\]
Note that relation (\ref{eq:21}) implies $Sing(X)=Sing(d\eta)\sub Sing(\eta)$.
Since $\eta$ represents $\G$ in homogeneous coordinates, we have
\[
Sing(\eta)=\{0\}\cup \Pi_{n-1}^{-1}(Sing(\G))\,\,.
\]

On the other hand, if $q\in Sing(\G)$ then $\Pi_{n-1}^{-1}(q)$ is a line $\ell$ through $0\in\C^n$ and, as we have seen in the proof 
of lemma \ref{l:210} the normal type of $\eta$ along $\ell$ coincides with the analytic type of a germ vector field representing $\G$ at $q$.
Therefore, if $\G\in\M(d)$ then $\ell\setminus\{0\}\sub K(\eta)$ and $d\eta(p)\ne0$, $\forall p\in\ell\setminus\{0\}$, so that $0$ is an isolated singularity of $d\eta$. Conversely, if $0$ is an isolated singularity of $d\eta$ then $\ell\setminus\{0\}\sub K(\eta)$ and $q$ is a singularity of Kupka type of $\G$.
\qed

The next definition can be found in \cite{ln3}.
\begin{definition}
Let $\eta$ be a germ at $0\in\C^n$ of integrable $(n-2)$-form with a singularity at $0$, $\eta(0)=0$. We say that $0$ is a {\it generalized Kupka singularity} (briefly: GK singularity) if it is an isolated singularity of $d\eta$, or equivalently, an isolated singularity of the rotational, $X=rot(\eta)$. We say that $0$ is a NGK (nilpotent generalized Kupka) singularity if $DX(0)$ is nilpotent.
\end{definition}
The following result is proved in \cite{ln3}:

\begin{theorem}\label{t:213}
Let $\eta$ be a germ at $0\in\C^n$ of an integrable $(n-2)$-form with a NGK singularity at $0$ and $X=rot(\eta)$. Then there is a holomorphic germ of vector field $Y$ at $0\in\C^n$ such that:
\begin{itemize}
\item[(a).] $\eta=i_Y\,i_X\,\nu$, where $\nu=dz_1\wedge...\wedge dz_n$.
\item[(b).] The eigenvalues of $DY(0)$ are all rational and positive and $tr(DY(0))<1$.
\end{itemize}
Furthermore, there exists a holomorphic coordinate system in which we can take $Y=S+N$, where $S$ is linear semi-simple, $N$ and $X$ are polynomial vector fields, $DN(0)$ is nilpotent and they satisfy $[S,N]=0$, $[N,X]=0$ and $[S,X]=\la\,X$, $\la=1-tr(S)>0$. 
\end{theorem}

\begin{rem}
{\rm Since the eigenvalues of $S$ are rational positive they can be written as $\frac{p_1}{q},...,\frac{p_n}{q}$, where $p_1,...,p_n,q\in\N$ and $p_1,...,p_n$ are relatively primes. Note that the number $\wt\la:=q\,\la=q-q.\,tr(S)$ is a positive integer. If we set $\wt{S}=q\,S$ then $[\wt{S},X]=\wt\la\,X$. This relation says that the vector field $X$ is quasi-homogeneous with weights $p_1,...,p_n$. 
In this case, we will say that the singularity is NGK of type $(p_1,...,p_n:\wt\la)$.

Observe that when the weights are $p_1=...=p_n=1$ then $\wt{S}=R$, the radial vector field, and $X$ has homogeneous coefficients of degree $d=\wt\la+1$. A consequence of Lemma \ref{l:211} is that if $\G\in\M(d)$ then the form $\eta$ that represents $\G$ in homogeneous coordinates has a NGK singularity of type $(1,...,1;d-1)$ at $0\in\C^n$.
A nilpotent singularity of this type will be called a NGK {\it conic} singularity of degree $d$.}
\end{rem}

\begin{rem}\label{r:215}
{\rm In the statement of theorem \ref{t:213} the vector field $Y$ such that $\eta=i_Yi_X\nu$ can be decomposed as $Y=S+N$, where $S$ is semi-simple and $[S,N]=0$, $[N,X]=0$ and $[S,X]=\la\,X$, $\la=1-tr(S)$. In fact, the vector field $N$ vanishes if we assume that $X$ satisfies a Zariski open condition (proposition 3 of \cite{ln3}).

In the case of a conic NGK singularity of degre $d\ge2$, $N$ is necessarily linear nilpotent and this condition is Zariski open and dense. In fact, in \cite{ln3} it is proved that there exists a Zariski open and dense subset $\U(d)\sub\fol(d;1,n-1)$ such that for any $\G\in\U(d)$, if $\eta$ represents $\G$ in homogeneous coordinates, $X=rot(\eta)$, $N$ is linear nilpotent and $[X,N]=0$ then $N=0$. We will use the notation $\M'(d):=\M(d)\cap\U(d)$ and $\M'_g(d)=\M_g(d)\cap\U(d)$.}
\end{rem}

Another result proved in \cite{ln3} is the persistence of nilpotent singularities under deformation.
Let $(\eta_t)_{t\in B_r}$ be a holomorphic family of integrable $(n-2)$-forms on an open set of $U\sub \C^n$, where $B_r=\{t\in\C^m\,;\,||t||<r\}$. 
Assume that $\eta_0$ has a nilpotent singularity of type $(p_1,...,p_n;\la)$ at some point $q\in U$ .

\begin{theorem}\label{t:216}
In the above situation, there exist $0<r'<r$ and a holomorphic map $Q\colon B_{r'}\to U$ with the following properties:
\begin{itemize}
\item{} $Q(0)=q$ and $Q(t)$ is a nilpotent singularity of type $(p_1,...,p_n;\la)$ of $\eta_t$ for all $t\in B_{r'}$.
\end{itemize}
\end{theorem}

Next, let us consider a holomorphic family $(\eta_t)_{t\in B_r}$ of  $(n-2)$-forms on a neighborhood of a closed ball $0\in\ov{B}\sub\C^n$ with boundary $\pa B$. Let us assume:
\begin{itemize}
\item{} $\eta_0$ has homogeneous coefficients of degree $d+1$ and represents a foliation $\G_0\in\M'(d)$.
\end{itemize}

In this case, we can write $\eta_0=\rho\,i_Ri_{X_0}\nu$, where $R$ is the radial vector field on $\C^n$, $X_0=rot(\eta_0)$, $\rho=1/(n+d-1)$ and $\nu=dz_1\wedge...\wedge dz_n$. As we have seen in lemma \ref{l:211} the vector field $X_0$ has an isolated singularity and $\eta_0$ a conic NGK singularity of degree $d$ at $0\in\C^n$.

Note also that the singular set of $\eta_0$ contains exactly $N=(d^n-1)/(d-1)$ straight lines through $0\in\C^n$, each line corresponding to a singularity of $\G_0$. 

Let $X_t=rot(\eta_t)$. Since $X_0|_{\pa B}$ doesn't vanishes, by taking a smaller $r$ if necessary, we can assume that $X_t|_{\pa B}$ also doesn't vanishes. In particular, $Sing(X_t)\cap \ov{B}$ is necessarily finite and so $cod(Sing(X_t))\ge3$.
By de Rham's division theorem (cf. \cite{DR}) there exists a holomorphic family of vector fields $(Y_t)_{t\in D_r}$, defined in a neighborhood of $\ov{B}$, such that $Y_0=\rho\,R$ and
\begin{equation}\label{eq:22}
\eta_t=i_{Y_t}i_{X_t}\nu\,\,,\,\,\forall\,t\in D_r\,\,.
\end{equation}
Given $q\in \ov{B}$ and $k\ge0$ we will denote by $j^k_q(\eta_t)$ the $k^{th}$-jet of $\eta_t$ at $q$. 
As a consequence of theorem \ref{t:216} and from the compactness of $\ov B$ we can state the following:

\begin{cor}\label{c:217}
In the above conditions, there exists $0<r'\le r$ and a holomorphic map $Q\colon B_{r'}\to B$, with $Q(0)=0$, and  such that:
\begin{itemize}
\item[(a).] $Sing(X_t)\cap\ov{B}=Sing(Y_t)\cap\ov{B}=\{Q(t)\}$, $\forall t\in B_{r'}$.
\item[(b).] $Q(t)$ is a conic NGK singularity of degree $d$ of $\eta_t$ for all $t\in B_{r'}$.
In parti-cular, $j^{\,d}_{Q(t)}(\eta_t)=0$ and $j^{\,d+1}_{Q(t)}(\eta_t)$, viewed as $(n-2)$-form with homogeneous coefficients, represents a foliation $\G_t\in\M'(d)$. In particular, the correspondence $t\in B_{r'}\mapsto \G_t\in\M'(d)$ is holomorphic.
\item[(c).] $DY_t(Q(t))=\rho\,R$ for all $t\in B_{r'}$, where $R$ is the radial vector field on $\C^n$. In particular, by Poincaré's linearization theorem $\rho^{-1}\,Y_t$ is holomorphically equivalent at $Q(t)$ to $R$.
\end{itemize}
\end{cor}

\begin{rem}\label{r:218}
{\rm We would like to observe that the foliation $\G_t\in\M'(d)$ of corollary \ref{c:217} appears when we blow-up the point $Q(t)$. Indeed, if we denote the blow-up by $\pi\colon(\wt{B},E)\to(\ov B,Q(t))$ then the divisor $E$ is biholomorphic $\p^{n-1}$ and $\pi^*(\fa_t)$ extends to $\wt{B}$, the complex manifold obtained after the blow-up. On the other hand, by (c) of corollary \ref{c:217}, the radial vector field is tangent to $\fa_t$, and its strict transform by $\pi$ is transverse to $E$. In fact, it can be verified that $\pi^*(\fa_t)|_E\simeq \G_t$.}
\end{rem}

{\bf Notation.} 
In the situation of remark \ref{r:218} we will say that $\fa_t$ {\it represents $\G_t$ near $Q(t)$}.

\begin{rem}
{\rm Let $Y_t$ and $X_t$ be as in (\ref{eq:22}), so that $\eta_t=i_{Y_t}i_{X_t}\nu$ and $d\eta_t=i_{X_t}\nu$.
From these relations we get
\[
L_{Y_t}\eta_t=i_{Y_t}\,d\eta_t+d\left(i_{Y_t}\,\eta_t\right)=\eta_t\,\,.
\]
As a consequence of the above relation, the singular set of $\eta_t$ is $Y_t$-invariant. In fact, $Sing(\eta_t)$ is the closure of $N$ orbits of $Y_t$. 
Since $Y_t$ is equivalent at $Q(t)$ to a multiple of the radial vector field, the closure of each orbit is a smooth curve containing $Q(t)$.
On the other hand, since we are supposing that $Y_0=\rho\,R$, by continuity of the solutions of the flow of $Y_t$ with $t\in B_{r'}$, if $r'$ is small enough then the orbits of $Y_t$ are transverse to $\pa B$ and have as unique adherent point the point $Q(t)$. Hence, we can conclude that $Sing(\eta_t)\cap\ov B=\bigcup_{j=1}^N\ell_j(t)$, where:
\begin{itemize}
\item[(a).] $\forall\, j\in\{1,...,N\}$, $\ell_j(t)$ is the union of an orbit of $Y_t$ in $\ov B$ with $Q(t)$. In particular, $\ell_j(t)$ is transverse to $\pa B$ and biholomorphic to a closed disc.
\item[(b).] $\forall\, i\ne j$, $\ell_i(t)\cap\ell_j(t)=\{Q(t)\}$.
\item[(c).] $\forall\,j$, there exists a $C^\infty$ isotopy $\Phi^j\colon B_{r'}\times\ov\d\to\ov B$ such that $\Phi_0^j(\ov\d)=\ell_j(0)$ and $\Phi^j_t(\ov\d)=\ell_j(t)$, $\Phi^j(t,z)=\Phi^j_t(z)$.
\end{itemize} } 
\end{rem}

\section{Proof of theorem \ref{teob}}\label{ss:3}

\subsection{Plan of the proof.}\label{ss:31}
We begin by proving the theorem in the case of two dimensional foliations. In \S\,\ref{ss:33} we will prove that $PB(\nu,d,2,n)$ is an irreducible component of $\fol(\Te;2,n)$, $\Te=(d+n-1)\nu-n+1$. Let us give an idea of the proof.

First of all, $PB(\nu,d,2,n)$ is an unirational irreducible algebraic subset of $\fol(\Te;2,n)$, because it is the closure in $\fol(\Te;2,n)$ of the set $\{f^*(\G)\,|\,f\in RM(n,n-1,\nu)\,,\,\G\in\fol(d,1,n-1)\}$. Let $Z$ be the (unique) irreducible component of $\fol(\Te;2,n)$ containing 
$PB(\nu,d,2,n)$. Since $PB(\nu,d,2,n)$ and $Z$ are irreducible it is sufficient to prove that there exists $\fa=f^*(\G)\in PB(\nu,d,2,n)$ such that for any germ of holomorphic one parameter family $(\fa_t)_{t\in (\C,0)}$ of foliations $\fa_t\in Z$ with $\fa_0=\fa$ then $\fa_t\in PB(\nu,d,2,n)$, $\forall\,t\in(\C,0)$.

We choose $\fa=f^*(\G)$, where $(f,\G)$ is a generic pair (see \S \ref{ss:23}), and $\G\in \M'_g(d)$ (see \S \ref{ss:21} and \S \ref{ss:25}).
Given the one parameter family $(\fa_t)_{t\in (\C,0)}$ with $\fa_0=f_0^*(\G_0)$, we will construct in \S\,\ref{ss:32} two holomorphic one parameter families $(f_t)_{t\in (\C,0)}$ and $(\G_t)_{t\in (\C,0)}$ of generic maps and foliations, such that $f_0=f$, $\G_0=\G$ and $f_t^*(\G_t)=\fa_t$ for all $t\in(\C,0)$, so that $\fa_t\in PB(\nu,d,2,n)$ for all $t\in (\C,0)$.
In the next section we will describe briefly how to find these families. 

A problem with the families $(f_t)_{t\in (\C,0)}$ and $(\G_t)_{t\in (\C,0)}$ that we will construct in \S \ref{ss:32} is that we cannot assert a priori that $\fa_t=F_t^*(\G_t)$, $\forall t\in(\C,0)$.
This fact will be proved in \S \ref{ss:33} in the case of two dimensional foliations. 
In \S\,\ref{ss:34} we will see how to reduce the case of foliations of dimension $\ge3$ to the case of two dimensional foliations.

\subsection{Construction of the families $(\G_t)_{t\in(\C,0)}$ and $(f_t)_{t\in(\C,0)}$.}\label{ss:32}

Recall that $Sing(\fa_0)$ contains $N$ smooth complete intersection curves $V_{q_1}$,...,$V_{q_N}$, such that:
\begin{itemize}
\item{} $V_{q_j}=\ov{f^{-1}(q_j)}$, $1\le j\le N$, where $\{q_1,...,q_N\}=Sing(\G_0)$.
\item{} $V_{q_i}\cap V_{q_j}=I(f_0)$, $\forall i\ne j$. In particular, $I(f_0)\sub V_{q_j}$, $\forall j$.
\item{} $V_{q_j}\setminus I(f_0)\sub K(\fa_0)$, $\forall j$.
\end{itemize}
In fact, it can be proved that $K(\fa_0)=\bigcup_jV_{q_j}\setminus I(f_0)$, but we will not use this essentially.
Using corollary \ref{c:217}, about deformations of conic NGK points we can state the following:

\begin{cor}
{\rm Set $I(f_0)=\{p_1,...,p_\rho\}$, $\rho=\nu^n$. Then there are holomorphic germs of curves $P_j\colon(\C,0)\to\p^n$, $1\le j\le \rho$, such that $P_j(0)=p_j$ and $P_j(t)$ is a conic NGK singularity of degree $d$.
Furthermore, for each $j\in\{1,...,\rho\}$ there exists a germ of holomorphic one parameter family of foliations $t\in(\C,0)\mapsto \G^j_t\in\M'(d)$  such that $\G^j_0=\G_0$ and $\fa_t$ represents $\G^j_t$ near $P_j(t)$.}
\end{cor}

{\bf Notation.} We will use the notation: $Con(\fa_t)=\{P_1(t),...,P_\rho(t)\}$.

\begin{rem}\label{r:32}
{\rm Since, in principle, we have $\rho=\nu^n$ different one parameter families of one dimensional foliations, $(\G^j_t)_{t\in (\C,0)}$,
$1\le j\le \rho$, we cannot assert a priori that, if $i\ne j$, then $\G^i_t$ is equivalent to $\G^j_t$ for any $t\in(\C,0)$. Indeed, this fact is true, but it will be a consequence of the final result.}
\end{rem}

{\bf Notation.} We will choose the family $(\G_t)_{t\in(\C,0)}$ of \S \ref{ss:31} as the family $(\G^1_t)_{t\in(\C,0)}$.
\begin{rem}\label{r:33}
{\rm  Since $\G_0\in\M_g(d)\sub\M(d)$ and $\M(d)$ is open, we can assert that $\G_t\in\M(d)$ for all $t\in(\C,0)$. However, we don't know if $\M_g(d)$ is open, but only Zariski generic in $\M(d)$. Therefore, we can't assert that $\G_t\in\M_g(d)$ for all $t\in(\C,0)$. On the other hand, we can assert that there exists a countable subset $C\sub(\C,0)$ such that $\G_t\in\M_g(d)$ for all $t\notin C$. We leave the details of this assertion for the reader.}
\end{rem}

\vskip.1in

Next we will describe briefly how to obtain the family of maps $(f_t)_{t\in(\C,0)}$ of \S \ref{ss:31}. The construction of this family will be done by using the deformation of the curves $V_{q_j}$, $1\le j\le N$, and a theorem of Sernesi \cite{Ser}.  
First of all, an easy consequence of proposition \ref{p:29} and corollary \ref{c:217} is the following:

\begin{lemma}\label{l:34}
{\rm There exist $N$ germs of $C^{\infty}$ isotopies $\phi^j\colon (\C,0)\times V_{q_j} \to \mathbb P^n$, $1\le j\le N$, such that if we denote $V_j(t):=\phi_{q_j}(\{t\}\times V_{q_j})$ then:
\begin{enumerate}
\item[(a)] $V_j(0)=V_{q_j}$ and $V_j(t) \setminus Con(\fa_t)$ is contained in the Kupka set of $\fa_t$, $\forall\,\, t\in (\C,0)$ and $\forall\,\,1\le j\le N$.  
\item[(b)] $Con(\fa_t) \subset V_j(t)$, $\forall\,\,1\le j\le N$ and $\forall\,\,t \in(\C,0)$. Moreover, if $i \neq j$ then $V_i(t)  \cap V_j(t) = Con(\fa_t)$, $\forall\,\,t \in (\C,0)$.
\end{enumerate}
In particular, $V_j(t)$ is an algebraic smooth curve of $\p^n$, $\forall\,\,1\le j\le N$ and $\forall\,\, t\in (\C,0)$.}
\end{lemma}

\begin{proof}  The argument is similar to \cite[lema 2.3.3, p.83]{ln} and uses essentially the local stability under deformations of the Kupka set and of $Con(\mathcal F_0)$ (see also \cite{clne}).
 \end{proof}
 
The map $f_{t}:\mathbb{P}^n-\to\mathbb{P}^{n-1}$, $f_{t} \in Gen\left(n,n-1,\nu\right)$ will be constructed in such a way that the curves $V_j(t)$, $1\le j\le N$, will be fibers of $f_t$, $\forall t\in(\C,0)$.
Since $d\ge2$ we have $\#Sing(\G_0)=N=(d^n-1)/(d-1)>n$, and we can assume that $\{q_1,...,q_n\}\sub Sing(\G_0)$, where $q_{1}=[1:0:...:0],...,q_{n}=[0:0:...:1]$. 

\begin{lemma}\label{l:35}
{\rm Let $(\mathcal{F}_{t})_{t \in (\C,0)}$ be as before; $\fa_0=f^*_0(\G_0)$. Then there exists a holomorphic germ of deformation
$({f}_{t})_{t \in(\C,0)}$ of $f_{0}$ in $Gen\left(n,n-1,\nu\right)$ such that:
\begin{enumerate}
\item[(i)] $Con(\mathcal F_t)=I(f_{t}), \forall\,\,t\in(\C,0)$.
\item[(ii)] $V_j(t)$ is a fiber of $f_t$, $\forall\,\,1\le j\le N$, $\forall\,\,t\in(\C,0)$.
\item[(iii).] The pair $(f_t,\G_t)$ is generic for all $t\in(\C,0)$.
\end{enumerate}}
\end{lemma}

\begin{proof}
Let $\tilde f_{0}=(F_0,...,F_{n-1}): \mathbb C^{n+1} \to\mathbb C^{n}$ be the homogeneous lifting of $f_{0}$. By the choice of $q_1,...,q_n$, the first $n$ curves $V_{q_1}$, $V_{q_2}$ ,..., and $V_{q_n}$ appear as the complete intersections $V_{q_i}=\Pi_n(F_{0}=F_{1}=...=\widehat{F_{i-1}}=...=F_{n-1}=0)$, where the symbol $\widehat{F_{j}}$ means omission of $F_{j}$ in the sequence. The remaining curves $V_{q_i}, i>n$ are also defined by $n-1$ polynomials in the ideal $\I_0:=\left<F_{0},\dots,F_{n-1}\right>$.

Now, we use Sernesi's stability criteria \cite[ sec. 4.6 235-236]{Ser}. It follows from lemma \ref{l:34} and Sernesi's criteria that for each $j\in \{1,...,N\}$ the curve $V_j(t)$ is also a complete intersection. Moreover, it is defined by an ideal of homogeneous polynomials of degree $\nu$, $\I^j_t=\left<G_{1\,t}^j,...,G_{n-1\,t}^j\right>$ such that each germ $t\in(\C,0)\mapsto G^j_{i\,t}\in\C[X_0,...,X_n]$, $1\le i\le n-1$, is holomorphic and moreover $\I^j_0\sub\I_0$. For instance, in the case $j=1$ we have $V_{q_1}=\Pi_n(F_1=...=F_{n-1}=0)$ and there are holomorphic families of homogeneous polynomials
$t\in(\C,0)\mapsto F_{i\,t}\in\C[X_0,...,X_n]$, $1\le i\le n-1$, such that $F_{i\,0}=F_i$ and 
\[
V_1(t)=\Pi_n(F_{1\,t}=...=F_{n-1\,t}=0)\,\,,\,\,\forall\,t\in(\C,0)\,\,.
\]
Similarly, there are holomorphic germs $t\in(\C,0)\mapsto F_{0\,t}\in\C[X_0,...,X_n]$, $t\in(\C,0)\mapsto\wt{F}_{2\,t}\in\C[X_0,...,X_n]$, ...,
$t\in(\C,0)\mapsto\wt{F}_{n-1\,t}\in\C[X_0,...,X_n]$ such that $F_{0\,0}=F_0$, $\wt{F}_{2\,0}=F_2$, ..., $\wt{F}_{n-1\,0}=F_{n-1}$ and
\[
V_2(t)=\Pi_n(F_{0\,t}=\wt{F}_{2\,t}=...=\wt{F}_{n-1\,t}=0)\,\,,\,\,\forall\,t\in(\C,0)\,\,.
\]

{\bf Notation.} We choose the family $(f_t)_{t\in(\C,0)}$ in such a way that $\wt{f}_t=(F_{0,t},...,F_{n-1\,t})$, where $\wt{f}_t$ is the homogeneous lifting of $f_t$, for all $t\in(\C,0)$. Note that $f_t$ is a generic map for any $t\in(\C,0)$.

\vskip.1in

Let us prove that $Con(\fa_t)=I(f_t)$, $\forall\,t\in(\C,0)$. First of all, $\Pi_n(F_{0\,t}=0)$, ...,$\Pi_n(F_{n-1\,t}=0)$ intersect multitransversely at $\nu^n$ points, because $f_t$ is generic for all $t\in(\C,0)$. Since $I(f_t)=\Pi_n(F_{0\,t}=...=F_{n-1\,t}=0)$ we get $\#\,I(f_t)=\nu^n$. On the other hand, $Con(\fa_t)\sub I(f_t)$, because $Con(\fa_t)=V_1(t)\cap V_2(t)=$
\[
=\Pi_n(F_{1\,t}=...=F_{n-1\,t}=0)\,\cap\,\Pi_n(F_{0\,t}=\wt{F}_{2\,t}=...=\wt{F}_{n-1\,t}=0)\,\sub\,I(f_t)\,\,.
\]
Finally $Con(\fa_t)=I(f_t)$ because $\#\,Con(\fa_t)=\nu^n=\#\,I(f_t)$.

It remains to prove that $V_j(t)$ is a fiber of $f_t$, $\forall\,t\in(\C,0)$. Here we use Noether's theorem about multitransversal intersections.
The fact that $Con(\fa_t)=I(f_t)$ is a multitransversal intersection and $Con(\fa_t)\sub V_j(t)$ imply that the ideal $\I_j(t)$ that defines $V_j(t)$ in homogeneous coordinates is contained in the ideal $\I_t:=\left<F_{0\,t},...,F_{n-1\,t}\right>$. This of course implies that $V_j(t)$ is a fiber of $f_t$.

Finally, the pair $(f_t,\G_t)$ is generic for all $t\in(\C,0)$, because $(f_0,\G_0)$ is generic and the set of generic pairs is open.
\end{proof}

\subsection{End of the proof of Theorem \ref{teob} in the case of two dimensional foliations}\label{ss:33} 
Let $(\fa_t)_{t\in(\C,0)}$ be the germ of deformation of $\fa_0=f_0^*(\G_0)$ of \S \ref{ss:31} and $(f_t,\G_t)_{t\in(\C,0)}$ be the germ of deformation of $(f_0,\G_0)$ obtained in \S \ref{ss:32}. Consider the holomorphic family of foliations $\left(\wt{\fa}_t\right)_{t\in(\C,0)}$ defined by $\wt{\fa}_t=f_t^*(\G_t)$, $\forall\,t\in(\C,0)$. Of course $\wt{\fa}_0=\fa_0$ and $\wt{\fa}_t\in PB(\nu,d,2,n)$, $\forall\,t\in(\C,0)$.

\begin{lemma}\label{l:36}
{\rm $\wt{\fa}_t=\fa_t$ for all $t\in(\C,0)$. In particular, $\fa_t\in PB(\nu,d,2,n)$, $\forall\,t\in(\C,0)$.}
\end{lemma}

\begin{proof}
The idea is to prove that $\wt{\fa}_t$ and $\fa_t$ have a common leaf $L_t$, $\forall\,t\in(\C,0)$. In particular, the foliations $\wt{\fa}_t$ and $\fa_t$ coincide in the Zarisk closure $\ov{L}_t^{\,Z}$ of $L_t$. 
On the other hand, we have seen in remark \ref{r:33} that there exists a germ of countable set $C\sub(\C,0)$ such that $\G_t\in \M_g(d)$ for all $t\notin C$. The fact that $\G_t\in\M_g(d)$ implies that the Zariski closure of any leaf of $\G_t$ is the whole $\p^{n-1}$. As we will see, this will imply that $\ov{L}_t^{\,Z}$ is the whole $\p^n$, $\forall\,t\notin C$. Since $C$ is countable this will finish the proof of lemma \ref{l:36}.

\vskip.1in

We begin by blow-up once at the $\rho=\nu^n$ points $P_1(t),...,P_\rho(t)$ of $Con(\fa_t)$. Let us denote by $M(t)$ the complex manifold obtained from this blow-up procedure, by $\pi_t\colon M(t)\to \p^n$ the blow-up map and by $E_1(t),...,E_\rho(t)$ the exceptional divisors obtained, where $\pi(E_j(t))=P_j(t)$, $1\le j\le \rho$. Denote by $V_j'(t)$ the strict transform of $V_j(t)$ by $\pi_t$.

\begin{rem}\label{r:37}
{\rm Since the pair $(f_t,\G_t)$ is generic and $I(f_t)=Con(\fa_t)=\{P_1(t),...,P_\rho(t)\}$, we can assert the following facts:
\begin{itemize}
\item[(I).] The map $f_t\circ\pi_t$ extends to a holomorphic map $f_t'\colon M(t)\to\p^{n-1}$.
Moreover, for any $1\le j\le \rho$, there exists a neighborhood $U_j$ of $P_j(t)$ such that $f_t'|_{U_j}\colon U_j\to\p^{n-1}$ is a submersion.
\item[(II).] The fiber $f_t'^{-1}(q)$, $q\in\p^{n-1}$ is the strict transform of $f_t^{-1}(q)$ by $\pi_t$. It is smooth near $E_j(t)$ and cuts $E_j(t)$ transversely in just one point, $\forall\,1\le j\le \rho$.
\item[(III).] $V_j'(t)$ is a smooth curve and $f_t'$ is a submersion in some neighborhood of $V_j'(t)$, $\forall\,1\le j\le \rho$.
\end{itemize}
Assertion (I) follows from the fact that $f_t$ is equivalent to the canonical projection $\Pi_{n-1}\colon\C^n\setminus\{0\}\to\p^{n-1}$ near $P_j(t)$, $1\le j\le \rho$. Of course, (I) $\implies$ (II) $\implies$ (III).}
\end{rem}

Let us denote by $\fa_t'$ and $\wt\fa_t'$ the strict transforms by $\pi$ of the foliations $\fa_t$ and $\wt\fa_t$, respectively.
Note that, for each $1\le j\le \rho$, the foliation $\wt{\fa}_t'|_{E_j(t)}$ is a foliation by curves on $E_j(t)\simeq\p^{n-1}$ equivalent to $\G_t$.
Similarly, $\fa_t'|_{E_j(t)}$ is a foliation by curves on $E_j(t)$, but we cannot assert that $\fa_t'|_{E_i(t)}$ is equivalent to
$\fa'_t|_{E_j(t)}$ if $i\ne j$. However, by the choice $\G_t$, we can assert that $\G_t$ is equivalent to $\fa_t'|_{E_1(t)}$.

On the other hand, by using (II) we can define a holomorphic map $\Phi_t\colon M(t)\to E_1(t)$, $1\le j\le \rho$, by
\[
\Phi_t(q):=f_t'^{-1}(f_t'(q))\cap E_1(t)\,\,,\,\,\forall\,q\in\,\p^n\,\,.
\]
Note that the fibers of $\Phi_t$ coincide with the fibers of $f_t'$. In fact, the maps $\Phi_t$ and $f_t'$ are equivalent, in the sense that there exists a biholomorphism $h\colon \p^{n-1}\to E_1(t)$ such that $\Phi_t=h\circ f_t'$. In particular, identifying $\G_t$ with $\fa_t'|_{E_1(t)}$ we can assert that
\[
\wt\fa'_t=\Phi_t^*(\G_t)\,\,.
\]

Now, we fix a singularity of $\G_t$, say $q(t)=q_1(t)$, with $V_1'(t)=\Phi_t^{-1}(q_1(t))$. Since $\G_t\in\M(d)$ it has $n-1$ analytic separatrices through $q_1(t)$, all smooth, say $\g_1(t),...,\g_{n-1}(t)$, and no other local analytic separatrix. Each separatrix $\g_j(t)$ is a germ of complex curve through $q(t)$ such that $\g_j(t)\setminus\{q(t)\}$ is contained in some leaf of $\G_t$.
If $\G_t\in\M_g(d)$ then its Zariski closure $\ov{\g_j(t)}^Z$ is $E_1(t)=\p^{n-1}$, because $\G_t$ has no proper algebraic invariant subset of positive dimension.
We fix one of these separatrices, say $\g_1(t)$. By construction the set $\Phi_t^{-1}(\g_1(t))$ satisfies the following property:
\begin{itemize}
\item[1.] It is $\wt\fa_t'$-invariant. In other words, $V_1'(t)\sub\Phi^{-1}_t(\g_1(t))$ and $\Phi^{-1}_t(\g_1(t))\setminus V_1'(t)$ is an open subset of some leaf of $\wt\fa_t'$.
\end{itemize}
We can assert also that:
\begin{itemize}
\item[2.] If $\G_t\in\M_g(d)$ then the Zariski closure $\ov{\Phi^{-1}_t(\g_1(t))}^{\,Z}$ is $\p^n$. This follows from the relation
\[
\Phi_t^{-1}\left(\ov{\g_1(t)}^{\,Z}\right)=\ov{\Phi_t^{-1}(\g_1(t))}^{\,Z}\,\,.
\]
\end{itemize}

{\bf Notation.} A {\it strip} $\Te$ around $V_1'(t)$ is a germ of smooth complex surface along $V_1'(t)$, containing $V_1'(t)$.
We say that a strip $\Te$ is a separatrix of $\fa_t'$ (resp. $\wt\fa_t'$) along $V_1'(t)$ if $\Te\setminus V_1'(t)$ is an open set of some leaf of $\fa'_t$ (resp. $\wt\fa_t'$).

Note that the set $\Phi^{-1}_t(\g_1(t))$ is a separatrix of $\wt\fa_t'$.
\begin{claim}\label{cl:38}
{\rm Let $\Te$ be a strip around $V_1'(t)$. Then for any $q\in\Te\cap E_1(t)$ the fiber $\Phi_t^{-1}(q)$ is contained in $\Te$.
In particular, if $\Te\cap E_1(t)$ is a separatrix of $\G_t$ then $\Te$ is a separatrix of $\wt\fa_t'$ along $V_1'(t)$.}
\end{claim}
\begin{proof}
Note that $\Te$ is transverse to $E_1(t)$. 
Consider a representative of $\Te$ transverse to $E_1(t)$, denoted by the same symbol. Since $\Phi_t$ is a submersion at the points of $V_1'(t)=\Phi_t^{-1}(q(t))$, there exists a holomorphic coordinate system around $q(t)\in E_1(t)\sub M(t)$, say $(x,y)\colon U\to\C^{n-1}\times\C$, $x=(x_1,...,x_{n-1})$, such that
\begin{itemize}
\item[(i).] $x(q(t))=0\in\C^{n-1}$ and $y(q(t))=0\in\C$.
\item[(ii).] $E_1(t)\cap U=(y=0)$ and $V_1'(t)\cap U=(x=0)$.
\item[(iii).] $\Phi_t(x,y)=(x,0)$.
\item[(iv).] $\Te\cap E_1(t)\cap U=(x_2=...=x_{n-1}=y=0)$. In other words, $\Te\cap E_1(t)$ is the $x_1$-axis in this coordinatte system.
\end{itemize}
Given $p(s):=(s,0,...,0)\in\Te\cap E_1(t)\cap U$ let 
$\sigma_s=\{(x,0)\in E_1(t)\cap U\,|\,x_1=s\}$ and $\Sigma_s:=\Phi_t^{-1}(\sigma_s)$.
Note that $\sigma_s$ is a hypersurface of $E_1(t)$ transverse to $\Te\cap E_1(t)$ at the point $p(s)\in \Te\cap E_1(t)$. This implies that $\Sigma_0$ is an hypersurface transverse to $\Te$ along $V_1'(t)$.

Let $B_r=\{(x,0)\in U\,;\,\,\,||x||<r\}$, the ball of radius $r>0$ in $E_1(t)$ centered at $0\in U$.
Since $\Phi_t$ is a submersion at the points of $V_1'(t)$ and the fibers of $\Phi_t$ are compact, there exists $\rho>0$ such that all points of $B_\rho$ are regular values of $\Phi_t$. This implies that, if $p(s)\in B_\rho$ then $\Sigma_s$ is a piece of smooth hypersurface of $M(t)$ containing $\Phi_t^{-1}(p(s))$. By compactness, taking a smaller $\rho$ if necessary, we can assume that $\Sigma_s$ is transverse to $\Te$,  and $\Te\cap \Si_s:=\theta_s$ is a compact curve of $M(t)$ contained in $\Te$, $\forall\,|s|=||p(s)||<\rho$.
Note that $p(s)\in\te_s$, $\forall\,|s|<\rho$.

Since $\theta_s$ is a compact curve, the set $\Phi_t(\theta_s)$ is a compact analytic subset of $B_\rho$. But this implies that $\Phi_t(\theta_s)$ is a point: $\Phi_t(\theta_s)=p(s)$. Therefore, $\theta_s$ is the fiber of $\Phi_t$ through $p(s)$.
Hence, if $|s|$ is small then the fiber $\Phi^{-1}(p(s))$ is contained in $\Te$. This finishes the proof of claim \ref{cl:38}.
\end{proof}

Now, the idea is to prove that $\fa'_t$ has some separatrix $\Te$ along $V_1'(t)$. Note that $\Te\cap E_1(t)$ is a separatrix of $G_t$ because $\G_t=\fa'_t|_{E_1(t)}$. By claim \ref{cl:38}, $\Te$ is also a separatrix of $\wt\fa_t'$, so that $\fa'_t$ and $\wt\fa_t'$ have a common leaf. This will conclude the proof of lemma \ref{l:36} and of theorem \ref{teob}.

\vskip.1in

In the construction of the separatrix $\Te$ of $\fa'_t$ along $V_1'(t)$ we use the fact that $V_1'(t)\sub K(\fa'_t)$, the Kupka set of $\fa'_t$.
Indeed, first of all $V_1'(t)\setminus \bigcup_j E_j(t)\sub K(\fa'_t)$ by lemma \ref{l:34}, because $\fa_t'=\pi_t^*(\fa_t)$.
On the other hand, the point of $V_1'(t)\cap E_j(t)$, $1\le j\le \nu^n$, is also in the Kupka set because $P_j(t)$ is a conic NGK singularity of $\fa_t$ (see \S \ref{ss:25}). We leave the details to the reader.

Since $V_1'(t)\sub K(\fa'_t)$ and $\fa'_t|_{E_1(t)}=\G_t$ there exist a finite covering $(U_{\a\,t})_{\a\in A}$ of $V_1'(t)$, a ball $B_\rho\sub E_1(t)$ around $q(t)$ and submersions $\phi_\a\colon U_\a\to B_\rho$ such that
\[
\fa'_t|_{U_\a}=\phi_a^*\left(\G_t|_{B_\rho}\right)\,\,.
\] 
If we fix some separatrix of $\G_t$ at $q(t)$, say $\g_1(t)$, then the set $\Te_\a(t):=\phi_\a^{-1}(\g_1(t))$ is a separatrix of $\fa'_t|_{U_\a}$ along $V_1'(t)\cap U_\a$, $\forall\,\a\in A$. On the other hand, they glue together in the sense that, if $U_\a\cap U_\be\ne\emp$ then 
\[
\Te_\a(t)\cap U_\a\cap U_\be=\Te_\be(t)\cap U_\a\cap U_\be\,\,.
\]
In fact, since $\fa_t'|_{U_\a\cap U_\be}=\phi_a^*(\G_t)=\phi_\be^*(\G_t)$ then $\phi_\a(\Te_\be(t)\cap U_\a\cap U_\be)$ is one of the separatrices of $\G_t$ through $q(t)$, $\g_j(t)$, $1\le j\le n-1$. On the other hand, it must be the separatrix $\g_1(t)$, because for $t=0$ we have $\fa'_0=\wt\fa'_0$ and $\Te_\be(t)$ is a deformation in $U_\be$ of the global separatrix $\Phi_0^{-1}(\g_1(0))$.
Hence, $\Te=\bigcup_\a\Te_\a(t)$ is a separatrix of $\fa'_t$ along $V_1'(t)$, as wished.
We leave the details for the reader.
\end{proof}

\subsection{Proof of theorem \ref{teob} in the case $k\ge3$}\label{ss:34}
As in the proof in \S\,\ref{ss:33} we will consider a germ of one parameter family of foliations $(\fa_t)_{t\in(\C,0)}$ on $\p^n$ such that $\fa_0=f^*(\G)\in PB(\nu,d,k,n)$, where the pair $(f,\G)$ is generic, and we will prove that $\fa_t\in PB(\nu,d,k,n)$ for all $t\in(\C,0)$. We will assume also that $\G\in \M_g(d)\sub\fol(d,1,m)$, as in the preceding proof. 
Since the dimension of $\fa$ is $k\ge3$, its codimension is $q=n-k$, the same codimension of $\G$ in $\p^m$, so that $m=q+1$.

Now we describe briefly the foliation $\fa$ near the indeterminacy locus $I(f)$. Fix $p\in I(f)$.
With an argument similar to the argument of \S\,\ref{ss:22} there exists a holomorphic coordinate system $(U,(x,y))$, where $x=(x_1,...,x_{m+1})\colon U\to\C^{m+1}$, $y\colon U\to\C^{n-m-1}$, such that
\[
f(x_1,...,x_{m+1},y)=[x_1:....:x_{m+1}]\,\,,\,\,\forall\,(x,y)\in U\,\,.
\]
If $\G$ is represented in homogeneous coordinates by the integrable $q$-form $\eta$ on $\C^{m+1}$, with homogeneous coefficients, then 
$\fa=f^*(\G)$ is represented in $U$ by $f^*(\eta)$, which has the same expression of $\eta$. In particular, $\fa|_U$ is equivalent to the product of a two dimensional homogeneous foliation on $\C^{m+1}$ by the regular foliation of dimension $n-m-1=k-2$ whose leaves are the levels $x=c$, $c\in\C^{m+1}$. Observe that the restriction $\fa|_{y=c'}$, $c'\in\C^{n-m-1}$, has a conic singularity at $(0,c')$.

\begin{definition}\label{d:39}
Let $\H$ be a germ of $k$-dimensional foliation on $(\C^n,0)$, $k\ge2$. We say that $\H$ has a {\it conic} singularity at $0\in\C^n$ if there exists a decomposition $\C^n=\C^{n-k+2}\times\C^{k-2}$ for which the restriction $\H|_{\C^{n-k-2}\times\{0\}}$ has a conic singularity at $(0,0)$.
We say that the singularity is NGK if the conic singularity is NGK (see \S\,\ref{ss:25}).
\end{definition}
Let $\H$ be as in definition \ref{d:39} and $\eta$ be a germ of integrable $(n-k)$-form defining $\H$. Clearly the condition in definition \ref{d:39} is equivalent to the fact that $\wt\eta:=\eta|_{\C^{n-k-2}\times\{0\}}$ has the first non-zero jet of conic type.
The singularity is NGK if, and only if, $\wt\eta$ is conic and $(0,0)$ is an isolated singularity of $d\wt\eta$ in the plane $\C^{n-k-2}\times\{0\}$. 

If $\fa=f^*(\G)$, where the pair $(f,\G)$ is generic then all points of $I(f)$ are conic NGK singularities of $\fa$. This is a consequence of the structure of local product described above, the fact that $\G\in\M(d)$ and of lemma \ref{l:211}.

Let $(\fa_t)_{t\in(\C,0)}$ be a germ of deformation of $\fa_0=\fa=f^*(\G)$.
We will denote the set of conic points of $\fa_t$ by $Con(\fa_t)$.
The proof of theorem \ref{teob} will be finished with two auxiliary results.

\begin{lemma}\label{l:310}
{\rm Let $(\fa_t)_{t\in(\C,0)}$ be a germ of one parameter deformation of $\fa_0=f^*(\G)$, where the pair $(f,\G)$ is generic. Then there exists a germ of $C^\infty$-isotopy $\Phi\colon I(f)\times(\C,0)\to\p^n$ such that $\Phi_0(I(f))=I(f)$ and $\Phi_t(I(f))=Con(\fa_t)$, $\forall t\in(\C,0)$, where $\Phi_t(z)=\Phi(z,t)$.}
\end{lemma}

\begin{lemma}\label{l:311}
{\rm Let $\eta$ be a germ at $0\in \C^n$ of integrable $q$-form, where $n\ge q+3$. Assume that $\eta$ has a conic NGK singularity of degree $d\ge2$ at $0\in\C^n$. Then there exists a germ of biholomorphism $\phi=(x,y)\colon(\C^n,0)\to(\C^{q+2}\times\C^{n-q-2},(0,0))$ such that $\phi_*(\eta)$ is dicritical, has homogeneous coefficients of degree $d+1$ and depends only on $x=(x_1,...,x_{q+2})$:
\[
\phi_*(\eta)=\sum_\si\,A_\si(x)\,dx_{\si(1)}\wedge...\wedge dx_{\si(q)}\,\,.
\]

In other words, the foliation $\fa_\eta$ is equivalent to a product of a singular foliation of dimension two on $\C^{q+2}$ by a regular foliation of dimension $n-q-2$.}
\end{lemma}
The proof of lemmae \ref{l:310} and \ref{l:311} will be done at the end.

\vskip.1in

{\bf Notation.} The form $\phi_*(\eta)$ represents a one dimensional foliation $\G$ on $\p^{q+1}$. The foliation $\G$ will be called the normal type of the foliation at the conic singularity.

\begin{rem}\label{r:312}
{\rm When the normal type is $NGK$ then Lemma \ref{l:311} implies that it is locally constant along the set of conic points.}
\end{rem}

Let us finish the proof of theorem \ref{teob}.
First of all, $Con(\fa_0)=I(f)$ is the complete intersection of $m+1$ hypersurfaces of degree $\nu$, so that $dim_\C(Con(\fa_0))=n-m-1=k-2\ge1$. 
On the other hand, by lemma \ref{l:310}, $Con(\fa_t)$ is a deformation of $Con(\fa_0)$, and so we can use Sernesi's theorem\, \cite{Ser}: there are holomorphic families of homogeneous polynomials of degree $\nu$, $F_{0\,t},...,F_{m\,t}$, that define $Con(\fa_t)$ in homogeneous coordinates, for all $t\in(\C,0)$. Each $F_t=(F_{0\,t},...,F_{m\,t})\colon \C^{n+1}\to\C^{m+1}$ defines a rational map $f_t\colon\p^n-\to\p^m$ such that $I(f_t)=Con(\fa_t)$.

Now, the normal type of $\fa_t$ along $Con(\fa_t)$ is locally constant, by remark \ref{r:312}. Since $Con(\fa_t)$ is a complete intersection, it is connected, so that the normal type is constant along $Con(\fa_t)$, for all $t\in(\C,0)$. Let $\G_t$ be the normal type of $\fa_t$ along $Con(\fa_t)$. The idea is to prove that $\fa_t$ is equivalent to $f_t^*(\G_t)$, for all $t\in(\C,0)$.

Let us define $f_t$ in such a way that $f_t^*(\G_t)=\fa_t$ for alll $t\in(\C,0)$. Fix a $q+2$ plane $H_o$ transverse to $Con(\fa_0)$ and fix a point $p\in H_o\cap Con(\fa_0)$. Since $p$ is a conic NGK point of $\fa_0|_{H_o}$, by theorem \ref{t:216}, there exists a holomorphic germ $Q\colon(\C,0)\to H_o$ such that $Q(0)=p$ and $Q(t)$ is a conic NGK singularity of $\fa_t|_{H_o}$. Since $t\mapsto Q(t)$ is holomorphic, after an automorphism of $\p^n$ that preserves $H_o$, we can assume that $Q(t)=p$ for all $t\in(\C,0)$.
Now, we blow up $H_o$ at $p$\, obtaining a divisor $E\simeq\p^{q+1}$ and a holomorphic one parameter family of foliations on $E$ that we can assume to be the family $(\G_t)_{t\in(\C,0)}$. Fix a ball $U$ around $p$ in $H_o$ such that $I(f_t)\cap H_o\cap U$ contains only the point $p$. If $U$ is small then we can assume that all fibers of $f_t$ cut $H_o\cap U$ in a smooth curve passing through $p$. The strict transform of this curve after the blow-up cuts $E$ in a unique point and this defines a rational map $\p^n-\to E\simeq\p^{q+1}$. This map is equivalent to $f_t$, so that we can assume that $f_t$ is constructed in this way.
Now, if we apply the argument of \S\,\ref{ss:33} we see that $f_t^*(\G_t)|_{H_o}=\fa_t|_{H_o}$, for all $t\in(\C,0)$.
The same argument can be applied to any $q+2$ plane $H$ transverse to $Con(\fa_0)$ to show that $\fa_t|_H=f_t^*(\G_t)|_H$, for all $t\in(\C,0)$.
This implies that $\fa_t=f_t^*(\G_t)$, $\forall t\in(\C,0)$, as the reader can check. This finishes the proof of theorem \ref{teob}.

\vskip.1in

{\it Proof of lemma \ref{l:310}.}
Lemma \ref{l:310} is a consequence of the stability of conic singularities, theorem \ref{t:216}. First of all, let us prove the local stability. 
Given $z\in Con(\fa_0)$ fix coordinates $(U,(x,y))$, $\phi=(x,y)\colon U\to Q_1\times Q_2\sub \C^{n-q-2}\times\C^{q+2}$ with $z\in U$ and $\phi(z)=(0,0)$, such that $U\cap Con(\fa_0)=(y=0)$.
Given $(x_o,0)\in Con(\fa_0)$ set $L_{x_o}:=\{(x,y)\in U\,|\,x=x_o\}\simeq Q_2$. Each $L_x$ is transverse to $Con(\fa_0)$, so that $\fa_0|_{L_x}$ has a conic NGK singularity at $(x,0)$. Applying theorem \ref{t:216} to the family of foliations $(\fa_t|_{L_x})_{(t,x)}$, viewed as a family of foliations on the open set $Q_2$ of $\C^{q+2}$, we obtain a holomorphic map $\psi\colon Q_1\times(\C,0)\to Q_2$ such that $\psi(x,0)=(x,0)$ and $\psi(x,t)$ is the unique conic singularity of $\fa_t|_{L_x}$ in $L_x$. This holomorphic local version implies that, locally "$Con(\fa_t)$ is a graph over $Con(\fa_0)$", and this implies the $C^\infty$ global version. We leave the details to the reader.
\qed

\vskip.1in

{\it Proof of lemma \ref{l:311}.}
According to definition \ref{d:39}, there exists a decomposition $\C^n=\C^{q+2}\times\C^{k-2}$ such that $\eta|_{(\C^{q+2},0)\times\{0\}}$ has a conic NGK singularity at $(0,0)$.
Let us fix some notations. We denote the coordinates in $\C^{q+2}\times\C^{k-2}$ as $(x,y)$, where $x=(x_1,...,x_{q+2})$ and $y=(y_1,...,y_{k-2})$. Given sequences $\a=(1\le\a_1<...<\a_s\le q+2)$ and $\be=(1\le\be_1<...<\be_r\le k-2)$ we set $\#\a=s$, $\#\be=r$ and
$
dx^\a\wedge dy^\be=dx_{\a_1}\wedge...\wedge dx_{\a_s}\wedge dy_{\be_1}\wedge...\wedge dy_{\be_r}
$.
With this notation a germ $\Te$ of holomorphic $q$-form can be written as
\begin{equation}\label{eq:31}
\Te=\sum_{\#\a+\#\be=q}\,A_{\a,\be}(x,y)\,dx^\a\wedge dy^\be\,\,,\,\,A_{\a,\be}\in\O_n\,\,.
\end{equation}
We will say that the $q$-form $\Te$ depends of $r$ variables $y$, where $0\le r\le k-2$, if it can be written as
\[
\Te=\sum_{\underset{1\le\be_j\le r}{\#\a+\#\be=q}}\,A_{\a,\be}(x,y_1,...,y_r)\,dx^\a\wedge dy^\be\,\text{, if $r>0$, or as}\,\,\Te=\sum_{\#\a=q}\,A_\a(x)\,dx^\a\,\,,
\]
if $r=0$. In other words, $\Te$ depends only of $x_1,...,x_{q+2},y_1,...,y_r$ and of $dx_1,...,dx_{q+2},dy_1,...,dy_r$.

The idea is to prove that if $\eta$ depends on $r$ variables $y$, where $1\le r\le k-2$, then there exists a germ of biholomorphism $\var\colon(\C^n,0)\to(\C^n,0)$, of the form $\var(x,y)=(\var_1(x,y),y)$, such that $\var^*(\eta)$ depends on $r-1$ variables $y$.
Of course, this implies lemma \ref{l:311}.
The induction step will be reduced to the following:

\begin{claim}\label{cl:313}
{\rm There exists a germ of vector field $Y$ of the form
\begin{equation}\label{eq:32}
Y=\frac{\pa}{\pa y_r}+\sum_{j=1}^{q+2}\,B_j\,(x,y_1,...,y_r)\,\frac{\pa}{\pa x_j}
\end{equation}
such that $i_Y\eta=0$ and $i_Y\,d\eta=0$.}
\end{claim}

We will prove claim \ref{cl:313} at the end. Let us finish the proof of the induction step by using it.

Let $\phi\colon(\C\times\C^{q+2+r},(0,0))\to(\C^{q+2+r},0)$ be the local flow of $Y$. The reader can check using (\ref{eq:32}) that $\phi$ is of the form
\[
\phi(t,x,y)=(\phi_1(t,x,y),y_1,...,y_{r-1},y_r+t)\,\,.
\]
Define
$\var\colon(\C\times\C^{q+2+r-1},(0,0))\to(\C^{q+2+r},0)$ as
\[
\var(t,x,y_1,...,y_{r-1})=\phi(t,x,y_1,...,y_{r-1},0)=(\phi_1(t,x,y_1,...,y_{r-1},0),y_1,...,y_{r-1},t)\,\,.
\]
It can be verified that $\var$ is a germ of biholomorphism and that $\var^*(Y)=\frac{\pa}{\pa t}$. In particular, if we set $\wt\eta=\var^*(\eta)$ then
\[
i_{\frac{\pa}{\pa t}}d\wt\eta=0\,\,\text{and}\,\,i_{\frac{\pa}{\pa t}}\wt\eta=0\,\,\implies\,\,L_{\frac{\pa}{\pa t}}\wt\eta=0\,\,.
\]
Since $i_{\frac{\pa}{\pa t}}\wt\eta=0$ the form $\wt\eta$ does not contain terms with $dt$. Since $L_{\frac{\pa}{\pa t}}\wt\eta=0$ the coefficients of $\wt\eta$ do not depend on $t$. Therefore, $\wt\eta$ depends only of $r-1$ variables $y$, as wished.

\vskip.1in

{\it Proof of Claim} \ref{cl:313}.
Fix coordinates $(x,y)\in(\C^{q+2}\times\C^{k-2},(0,0))$ such that $\eta|_{(y=0)}$ has a conic NGK singularity.
We begin proving that we can assume that the set of conic NGK singularities of $\eta$, $Con(\eta)$, is $\{(x,y)\,|\,x=0\}$.

Indeed, fix coordinates $(x,y)=(x_1,...,x_{q+2},y_1,...,y_{k-2})\in\C^{q+2}\times\C^{k-2}$ such that $\eta|_{(y=0)}$ has a conic NGK singularity at $(0,0)$.
Given $y_o\in (\C^{k-2},0)$ set $\eta_{y_o}:=\eta|_{(y=y_o)}$.
If $\eta$ is written as in (\ref{eq:31}) then
\[
\eta_{y}=\sum_{\#\a=q}\,A_\a(x,y)\,dx^\a\,\,,\,\,\forall\,y\in(\C^{k-2},0)\,\,.
\]
According to the definition, $\eta_0$ has a conic NGK singularity at $0\in\C^{q+2}$.  We can consider $y\mapsto\eta_y$ as a holomorphic family of integrable $q$-forms. In this case, theorem \ref{t:216} implies that there exists a holomorphic germ $Q\colon(\C^{k-2},0)\to (\C^{q+2},0)$ such that $Q(0)=0$ and $Q(y)$ is the unique conic NGK singularity of $\eta_y$. In particular, $Con(\eta)=\{(Q(y),y)\,|\,y\in(\C^{k-2},0)\}$. Let $\phi$ be the germ of biholomorphism defined by $\phi(x,y)=(x+Q(y),y)$. Since $\phi^{-1}(Con(\eta))=(x=0)$, we get
\[
Con(\phi^*(\eta))=\phi^{-1}(Con(\eta))=(x=0),
\] 
which proves the assertion. From now on we will assume that $Con(\eta)=(x=0)$.

Let us assume that $\eta$ depends on $r$ variables $y$, where $1\le r\le k-2$.
Given $y_o\in{(\C^r,0)}$ let $\eta_{y_o}$ be as before:
\[
\eta_{y_o}=\eta|_{(y=y_o)}=\sum_{\#\a=q}\,A_\a(x,y_o)\,dx^\a\,\,.
\]
We can consider $d\eta_{y_o}$ as a $(q+1)$-form on $(\C^{q+2},0)$. In particular, there exists a holomorphic vector field
\begin{equation}\label{eq:33}
X=\sum_{j=1}^{q+2}\,X_j(x,y_o)\,\frac{\pa}{\pa x_j}
\end{equation}
such that $d\eta_{y_o}=i_X\,\nu_o$, where $\nu_o=dx_1\wedge...\wedge dx_{q+2}$. We consider $X$ as a germ of vector field on $(\C^n,0)$.
Since $d\eta$ is integrable, the fact that $i_X\,d\eta_{y}=0$ for all $y\in(\C^r,0)$ implies that $i_X\,d\eta=0$. The integrability of $\eta$ implies that $i_X\eta=0$.

\begin{rem}\label{r:314}
{\rm Note that $Sing(X)=Con(\eta)=(x=0)$.}
\end{rem}

\vskip.1in

From now on, we will write $y=(\wt{y},y_r)$, where $\wt{y}=(y_1,...,y_{r-1})$.
Given $\wt{y}_o=(y_{1\,o},...,y_{r-1\,o})\in(\C^{r-1},0)$ fixed, set $\Si_{\wt{y}_o}=\{(x,\wt{y},y_r)\,|\,\wt{y}=\wt{y}_o\}$ and $\eta_{\wt{y}}:=\eta|_{\Si_{\wt{y}}}$, so that $d\eta_{\wt{y}}=d\eta|_{\Si_{\wt{y}}}$. It follows from the above notations that there exists a germ of $q$-form $\te$, of the type
\[
\te=\sum_{\#\a=q}\,C_\a(x,y)\,dx^\a
\]
such that
\[
d\eta_{\wt{y}}=i_X\,\nu_o+\te\wedge dy_r\,\,.
\]
From $i_X\,d\eta=0$ we obtain $i_X\,d\eta_{\wt{y}}=0$, for all $\wt{y}\in(\C^{r-1},0)$. In particular,
\[
i_X d\eta_{\wt{y}}=i_X(\te\wedge dy_r)=i_X\te\wedge dy_r=0\,\,\implies\,\,i_X\te=0\,\,.
\]
Now, we use de Rham's division theorem in the parametric form, considering $\te$ and $X$ depending of the parameter $y\in(\C^r,0)$. Since $Sing(X)=(x=0)$, for $y\in(\C^r,0)$ fixed then $X$ has an isolated singularity at $0\in\C^{q+2}$. De Rham's theorem implies that there exists a $q+1$ form $\mu$ such that $\te=i_X\mu$, where
\[
\mu=\sum_{\#\a=q+1}D_\a(x,y)\,dx^\a\,\,.
\] 
Hence, we can write
\[
d\eta_{\wt{y}}=i_X(\nu_o+\mu\wedge dy_r)\,\,.
\]
The reader can verify that the $(q+2)$-form $\nu_o+\mu\wedge dy_r$ can be written as
\[
\nu_o+\mu\wedge dy_r=i_Y\,\left(dy_r\wedge dx_1\wedge...\wedge dx_{q+2}\right),
\]
where $Y$ is as in (\ref{eq:32}). On the other hand, $i_Y\,(\nu_o+\te\wedge dy_r)=0$ implies that $i_Y\,(d\eta_{\wt{y}})=0$ for all $\wt{y}\in(\C^{r-1},0)$.
Since $\eta$ and $d\eta$ are integrable, this implies that $i_Y\,d\eta=0$ and $i_Y\eta=0$, which proves claim \ref{cl:313} and lemma \ref{l:311}.
\qed

\section{Foliations that can be characterized by their singular set}

In \cite{cln1}, the authors, have proved that a foliation $\mathcal F$ on $\mathbb P^n, n\geq 3$ whose Kupka set $K(\mathcal F)$ contains a codimension two smooth irreducible component, say $\Lambda$, which is a complete intersection, has a rational first integral. Moreover, in this case, $K(\mathcal F)=\Lambda$. An analogous result for a particular class of pull-back foliations, was stated in \cite[Theorem B, p.709]{clne}.  Its worth point out that in \cite{oma}, the authors generalized the results of \cite{cln1} for foliations on $\mathbb{P}^n$ in any codimension. In particular they also obtained the same result contained in  \cite{cln1} for two dimensional foliations.   
Using Lemma \ref{l:36} a similar result to \cite[Theorem B, p.709]{clne} can be stated for two dimensional foliations $\mathcal F$ on $\mathbb P^n, $ which are pull-back of foliations on $\mathbb P^{n-1}$. 
\par Recall from Definition \ref{generic} the concept of a generic map.
\par Let $f  \in RM\left(n,n-1,\nu\right)$ be a generic map, $I(f)$ its indeterminacy locus and $\mathcal F$ a two dimensional foliation on $\mathbb P^n$, $n\geq3$. Consider the following properties:

\begin{center}
\begin{minipage}{10cm}
$\mathcal{P}_1:$ The foliation $\fa$ has a two dimensional conic NGK singularity of the type $(1,\dots,1,d-1)$ at any point of $I(f)$.
We assume also that at some point $p\in I(f)$ the normal type is a foliation $\G\in \M_g(d)$, where $d\ge2$. In particular, the Zariski closure of any leaf of $\G$ is all $\p^{n-1}$.
\end{minipage}
\end{center}

\begin{center}
\begin{minipage}{10cm}
$\mathcal{P}_2:$ There exists a fibre $f^{-1}(q)=V(q)$ such that $V(q)=\ov{f^{-1}(q)}\backslash I(f)$ is contained in the Kupka-Set of $\mathcal F$.
\end{minipage}
\end{center}

\begin{center}
\begin{minipage}{10cm}
$\mathcal{P}_3:$ $V(q)$ has transversal type $Y$, where $Y$ is a germ of vector field on $({\mathbb C^{n-1},0})$ of a hyperbolic singularity.
\end{minipage}
\end{center}

Lemma \ref{l:36} allows us to prove the following result:

\begin{main}\label{teoc}
In the conditions above, if properties $\mathcal{P}_1$, $\mathcal{P}_2$, and $\mathcal{P}_3$ hold then $\mathcal F$ is a pull back foliation, $\mathcal {F}= f^{*}(\mathcal{G})$, where $\mathcal{G}$ is the foliation of $\mathcal{P}_2$.  
\end{main} 

The idea of the proof is to consider the foliation $\wt\fa=f^*(\G)$ and use the argument of lemma \ref{l:36} to prove that $\wt\fa=\fa$.

An analogous result can be proved in the case of foliations of dimension greater than two.
Let $f\in RM(n,m,\nu)$ be generic, where $n\ge m+2$. In this case, its indeterminacy locus $I(f)$ is a smooth irreducible algebraic set of dimension $n-m-1\ge1$. Let $\fa$ be a dimension $k=n-m+1\ge3$ foliation with the following properties:

\begin{center}
\begin{minipage}{10cm}
$\wt{\mathcal{P}}_1:$ The set of conic points of $\fa$, $Con(\fa)$, contains $I(f)$. Moreover, the normal type of $\fa$ along $I(f)$ is NGK.
\end{minipage}
\end{center}

In this case, lemma \ref{l:311} implies that the normal type is constant along $I(f)$ (see remark \ref{r:312}).
Call this normal type $\G$. Since $f\colon \p^n-\to\p^m$, $\G$ is a foliation on $\p^m$.

\begin{center}
\begin{minipage}{10cm}
$\wt{\mathcal{P}}_2:$ The normal type $\G$ is generic in the sense of \S\,\ref{ss:21}: $\G\in\M_g(d)$, $d\ge2$.
\end{minipage}
\end{center}

\begin{center}
\begin{minipage}{10cm}
$\wt{\mathcal{P}}_3:$ There exists a fibre $f^{-1}(q)=V(q)$ such that $V(q)=\ov{f^{-1}(q)}\backslash I(f)$ is contained in the Kupka-Set of $\mathcal F$.
\end{minipage}
\end{center}

\begin{center}
\begin{minipage}{10cm}
$\wt{\mathcal{P}}_4:$ $V(q)$ has transversal type $Y$, where $Y$ is a germ of vector field on $({\mathbb C^m,0})$ of a hyperbolic singularity.
\end{minipage}
\end{center}

\vskip.1in

With the same arguments of \S\,\ref{ss:34} it is possible to prove the following:

\begin{main}\label{teod}
In the above conditions, if properties $\wt{\mathcal{P}}_1$, $\wt{\mathcal{P}}_2$,  $\wt{\mathcal{P}}_3$ and $\wt{\mathcal{P}}_4$  hold then $\mathcal F$ is a pull back foliation, $\mathcal {F}= f^{*}(\mathcal{G})$.  
\end{main}

Motivated by theorems \ref{teob}, \ref{teoc} and \ref{teod}, we would like to state the following problems:

\begin{prob}\label{pr:1}
{\rm Is there a generalization of theorem \ref{teob} in the case of pull-backs by branched maps (not generic) like in \cite{WS}?}
\end{prob}

\begin{prob}\label{pr:2}
{\rm Let $\fa$ be a $k$-dimensional foliation on $\p^n$, where $2\le k<n$. Assume that $\fa$ has a conic NGK singularity. Is it true that $\fa$ is a pull-back foliation?}
\end{prob}

In fact, we don't know any example of foliation of dimension $\ge2$ having a conic singularity which is not a pull-back.
Another problem, motivated by problem \ref{pr:2}, is the following:

\begin{prob}\label{pr:3}
{\rm Is there a holomorphic foliation of dimension $\ge2$, on a complex connected manifold, having two or more conic NGK singularities of different normal types?}
\end{prob}


\subsubsection*{Acknowledgments:} We are deeply grateful to D. Cerveau and T. Fassarella for the discussions, suggestions and comments. We would also like to thank E. Goulart for the corrections in the manuscript. This work was developed at IRMAR(Rennes, France) and was supported by Capes-Brasil, process number (9814-13-2).


\bibliographystyle{amsalpha}

\end{document}